\documentclass{amsart}

\usepackage{amssymb}
\usepackage[all]{xy}

\usepackage{pdfsync}

\title{On colimits and elementary embeddings}
\author{Joan Bagaria}

\address{ICREA (Instituci\'o Catalana de Recerca i Estudis Avan\c{c}ats) and
Departament de L\`ogica, Hist\`oria i Filosofia de la Ci\`encia, Universitat de Barcelona. 
Montalegre 6,
08001 Barcelona, Catalonia (Spain).}
\email{joan.bagaria@icrea.cat}
\thanks{The research work of the first author was partially supported by the Spanish 
Ministry of Science and Innovation under grants MTM2008-03389 and MTM2011-25229,
and by the Generalitat de Catalunya (Catalan Government) under grant 2009~SGR~187. The research of the second author was supported by the 
Heilbronn Institute for Mathematical Research at the University of Bristol,
the INFTY Network Short Visit Grant 2921 and INFTY Exchange Grant 2580,
and a Japanese Society for the Promotion of Science (JSPS)
Postdoctoral Fellowship
for Foreign Researchers and JSPS Grant-in-Aid for Scientific Research 
No. 23 01765.}
\author{Andrew Brooke-Taylor}
\address{Group of Logic, Statistics \& Informatics,
Graduate School of System Informatics,
Kobe University.
Rokko-dai 1-1,
Nada, Kobe,
657-0013, Japan.}
\email{andrewbt@gmail.com}


\newcommand{\al}{\alpha}
\newcommand{\ga}{\gamma}
\newcommand{\ka}{\kappa}
\newcommand{\la}{\lambda}
\renewcommand{\phi}{\varphi}
\newcommand{\D}{\mathcal{D}}
\newcommand{\restr}{\upharpoonright}
\newcommand{\Power}{\mathcal{P}}
\newcommand{\of}{\circ}
\newcommand{\sat}{\vDash}
\newcommand{\Id}{\textrm{Id}}
\newcommand{\into}{\hookrightarrow}
\newcommand{\st}{\,|\,}
\DeclareMathOperator{\Colim}{Colim}
\DeclareMathOperator{\crit}{crit}
\DeclareMathOperator{\Hom}{Hom}
\DeclareMathOperator{\Mor}{Mor}
\DeclareMathOperator{\Obj}{Obj}

\newcommand{\Set}{\mathbf{Set}}
\DeclareMathOperator{\Str}{\bf Str}
\DeclareMathOperator{\Mod}{\bf Mod}

\newcommand{\calA}{\mathcal{A}}
\newcommand{\calC}{\mathcal{C}}
\newcommand{\calI}{\mathcal{I}}
\newcommand{\calK}{\mathcal{K}}
\newcommand{\calL}{\mathcal{L}}
\newcommand{\calU}{\mathcal{U}}

\newcommand{\ZZ}{\mathbb{Z}}

\newcommand{\K}{\Str\Sigma}

\newcommand{\Vopenka}{Vop\v{e}nka}

\newcommand{\bPi}[1]{\Pi_{#1}}

\newtheorem{thm}{Theorem}
\newtheorem{lemma}[thm]{Lemma}
\newtheorem{defn}[thm]{Definition}

\newtheorem{eg}[thm]{Example}
\newtheorem{prop}[thm]{Proposition}
\newtheorem*{open}{Open Problem}

\begin{document}
\begin{abstract}
We give a sharper version of a theorem of Rosick\'y, Trnkov\'a and Ad\'amek \cite{RTA:UPLPC}, and a new proof of a theorem of Rosick\'y \cite{Ros:MDCM}, 
both about colimits in categories of structures. Unlike the original proofs, which use category-theoretic methods, we use set-theoretic arguments involving elementary embeddings given by large cardinals  such as $\alpha$-strongly compact and $C^{(n)}$-extendible cardinals.

\end{abstract}

\maketitle

\section{Introduction}

Many problems in category theory, homological algebra, and homotopy theory have been shown to be set-theoretical, involving the existence of large cardinals. For example, the problem of the existence of rigid classes in categories such as graphs, or metric spaces or compact Hausdorff spaces with continuous maps, which was studied by the Prague school in the 1960's  turned out to be equivalent to the  large cardinal principle now known as Vope\v{n}ka's Principle (VP) (see \cite{Kan:THI}). Another early example is John Isbell's 1960 result that $\mathbf{Set}^{op}$ is bounded if and only if there is no proper class of measurable cardinals. A summary of these and similar results can be found in the monograph \cite{PT:Comb}. In the 1980's, many statements in category theory  previously known to hold under the assumption of VP were shown to be actually equivalent to it. The following is a small sample of such statements (see \cite{AdR:LPAC} for an excellent survey, with complete proofs,  of these and many other equivalence  results):
\begin{enumerate}
\item The category $\mathbf{Ord}$ of ordinals cannot be fully embedded into the category $\mathbf{Gra}$ of graphs.
\item A category is locally presentable if and only if it is complete and bounded.
\item A category is accessible  if and only if it is bounded and has $\lambda$-directed colimits for some regular cardinal $\lambda$.
\item Every subfunctor of an accessible functor is accessible.
\item For every full embedding $F:\mathcal{A}\to \mathcal{K}$, where $\mathcal{K}$ is an accessible category, there is a regular cardinal $\lambda$ such that $F$ preserves $\lambda$-directed colimits.
\end{enumerate}

Even though each one of these statements is equivalent to VP, their known proofs from VP use category-theoretic, rather than set-theoretic,  methods and arguments. 

Recently, new equivalent formulations of VP in terms of elementary embeddings have been used in \cite{BCMR:DOCACS} to improve on previous results in category theory and homotopy theory.
For the reader unfamiliar with them, 
an \emph{elementary embedding} $j:V\to M$ is a function that 
preserves \emph{all} formulas: for every formula $\phi$ with parameters
$a_1,\ldots,a_n$ in $V$, $V\sat\phi(a_1,\ldots,a_n)$ if and only if
$M\sat\phi(j(a_1),\ldots,j(a_n))$, where $\sat$ denotes the model satisfaction 
relation.  Many of the strongest large cardinal axioms are most naturally
expressed in these terms, with the pertinent cardinal being the 
\emph{critical point}
$\crit(j)$ of the embedding, that is, the least cardinal $\ka$ for which
$j(\ka)\neq\ka$.

In \cite{BCMR:DOCACS}, these elementary embedding formulations are used
with a set-theoretic analysis to improve previous results by
showing that
much weaker large cardinal assumptions suffice for them. 
Specifically,
the necessary large cardinal hypothesis in each case depends on the complexity
of the formulas defining the categories involved, in the sense of the
L\'evy hierarchy. For example, the statement that, for a class
$\mathcal{S}$ of morphisms in an accessible category~$\mathcal{C}$, the orthogonal
class of objects $\mathcal{S}^{\perp}$ is a small-orthogonality class
is provable in ZFC if
$\mathcal{S}$ is $\Sigma_1$, it follows from the existence of a proper class
of supercompact cardinals if $\mathcal{S}$ is $\Sigma_2$, and from the
existence of a proper class of $C^{(n)}$\nobreakdash-extendible cardinals if $\mathcal{S}$ is 
 $\Sigma_{n+2}$
for $n\ge 1$. These cardinals form a hierarchy, and VP is equivalent to the existence of $C^{(n)}$\nobreakdash-extendible cardinals for all~$n$ (see also \cite{Ba:CC}).
As a consequence, the existence
of cohomological localizations of simplicial sets, a long-standing open problem in algebraic
topology solved in \cite{CSS:CCAG} assuming VP, follows just from the existence of sufficiently large supercompact cardinals.

In this paper we continue this programme of giving sharper versions of 
results in category theory, in this case results about colimits in accessible
categories, 
and more generally in categories of structures in infinitary languages. 
What is different in our work, however, is that we use the elementarity of
the embeddings in a strong way.  In previous work in this context, 
once one has obtained
an elementary embedding it has generally been used 
as little more than a convenient homomorphism.  By contrast,
we shall make great use of elementarity to make our proofs work,
as we move between various set-ups and their images under the embedding.
It is our belief that similar ``strong uses'' of elementarity will lead
to many improvements and new results in the area.

The first main result we consider (Theorem~\ref{ModStr} below) is due to 
J. Rosick\'y {\cite[Theorem~1 and Remark~1~(2)]{Ros:MDCM}}, which in turn generalizes an earlier result of Richter \cite{Ric:LKR}. 
The proof given by  Rosick\'y uses 
atomic diagrams, explicit ultraproducts, 
and something called ``purity''; our proof using elementary embeddings
and ideas from the original paper of Richter seems much cleaner.  

The second one we consider is item (5) in the list given above of statements equivalent to VP, a result of Rosick\'y, Trnkov\'a and Ad\'amek  from \cite{RTA:UPLPC} (see also \cite[Theorem~6.9]{AdR:LPAC}). 
In Theorem~\ref{thm2} below 
we prove a result that is simultaneously more general and
more refined, using appropriate  fragments of VP. 
Specifically, we use  $C^{(n)}$-extendible cardinals, with $n$ determined
by the complexity of the definitions of the cateogries involved. 
Further, we are able to show (Theorem \ref{reversal}) that the $C^{(n)}$-extendible cardinals are necessary in an almost, but unfortunately not exactly, level-by-level equivalence.

We would like to remark that while the proofs of the two theorems in the original papers are quite different from each other in their methods, our new proofs turn out to be quite similar. We think this shows that our set-theoretic arguments using elementary embeddings are more generally applicable  and  more natural in this context.

\section{Preliminaries}
\label{preliminaries}
As is standard, we denote by $V$ the universe of all sets.
For all of our elementary embeddings $j:V\to M$, 
$M$ is a class of $V$, so we can in effect treat $j$ as a functor from 
$V$ to $V$ 
(but note that it is only \emph{elementary} as a map with codomain $M$).
The notation $j``\D$ ($j$ point-wise on $\D$) denotes the diagram
consisting of the objects $j(d)$ for $d$ an object of $\D$, 
and the morphisms $j(f)$
for $f$ a morphism of $\D$. Note in particular that this will generally be
different from $j(\D)$, that is, $j$ applied to the diagram $\D$ 
as a whole: $j(\D)$ will generally contain many more objects and
morphisms than just those that are images of objects or morphisms from $\D$,
that is, those in $j``\D$.


Throughout the paper an important role will be played by the category
$\Str\Sigma$ of all $\Sigma$-structures for a signature $\Sigma$.  Here
a \emph{signature} is a set of function and relation symbols with associated
arities; formally a $\Sigma$-structure is something of the form 
$\langle A,\calI\rangle$ where $A$ is a set (the underlying set of the
structure) and $\calI$ is a function with domain $\Sigma$ 
(the interpretation function) such that for each $n$-ary relation symbol $R$, 
$\calI(R)\subseteq A^n$ (considered to be the set of tuple where $R$ holds),
and for each $n$-ary function symbol $f$,
$\calI(f)$ is a function from $A^n$ to $A$.
We abuse notation identifying $A$ with $\langle A,\calI\rangle$,
and we shall often write $R^A$ or $f^A$ for $\calI(R)$ or $\calI(f)$ 
respectively.

Associated to any signature, we have a language of formulas
with functions and relations from that signature; see for example
\cite[Sections 5.2 and 5.24]{AdR:LPAC}.
Note that 
we do not constrain ourselves to signatures in which all of the arities are 
finite;
in line with this, we also consider infinitary languages.
In this setting, $\calL_\la(\Sigma)$ denotes the language in which in the
definition of formulas we allow 
conjunctions, disjunctions, and universal and existential
quantifications of size less than $\la$.
Thus, $\calL_\omega(\Sigma)$ is the usual language over $\Sigma$ with
only finitary conjunctions, disjunctions and quantifications;
note however that if $\Sigma$ contains any symbols of infinite arity,
then $\calL_\omega(\Sigma)$ cannot have any fully quantified sentences
involving those symbols.
The notion of a structure $A$ satifying a formula $\phi$ in $\calL_\la(\Sigma)$,
denoted $A\sat\phi$,
is defined in line with the expected meaning --- see for example
\cite[Sections 5.3 and 5.26]{AdR:LPAC}.

Whilst $\Sigma$ may be infinite and infinitary, its basic logical role
means that we generally do not want it to be affected by the 
elementary embeddings we employ, which are only elementary for the 
language of set theory, 
$\calL_\omega(\{\in\})$.
Indeed in Section~\ref{GenCatssection}
we make assumptions to this effect.  However, in Section~\ref{StrModsection},
we can in fact handle a $\Sigma$ large enough to be changed by the embedding,
so long as the arity of each individual symbol is small enough to be unaffected.
To this end, let us
call a signature $\Sigma$ $\lambda$-ary if every symbol in $\Sigma$ has arity
strictly less than $\lambda$.

\begin{defn}\label{Redj}
Suppose $j$ is an elementary embedding from $V$ to $M$ 
with critical point $\ka$,
$\Sigma$ is a $\ka$-ary signature,
and $C$ is a $j(\Sigma)$-structure.  Then $C_{\Sigma}$ is the
$\Sigma$-structure obtained by reducing $C$ to $C'$ over signature
$j``\Sigma$, and then
considering the interpretation of $j(R)$ in $C'$ to be the interpretation
of $R$ in $C_{\Sigma}$, for any symbol $R$ in $\Sigma$.
\end{defn}

Here ``reducing'' is in the model-theoretic sense of simply ``forgetting''
those function and operation symbols in $j(\Sigma)$ but not $j``\Sigma$ 
(and leaving the underlying set unchanged).
Thus, the structure $\langle C, R^C \st R\in j(\Sigma)\rangle$
reduces to $\langle C, R^C \st R\in j``\Sigma\rangle$, and then it is simply
a matter of relabelling the indices to consider this to be
$\langle C, j(S)^C \st S\in\Sigma\rangle=C_\Sigma$.

We give some very basic lemmas about this notion.
For all of them,
take $j$, $\ka$, and $\Sigma$ as in 
Definition~\ref{Redj}.

\begin{lemma}\label{RedTMod}
For any $\lambda<\ka$ and 
any theory $T$ for the language $\calL_\lambda(\Sigma)$,
if $C$ is a model of $j(T)$, then $C_\Sigma$ is a model of $T$.
\end{lemma}
\begin{proof}
This is immediate from the recursive definition of $\sat$, as presented
for example in \cite[Section~5.26]{AdR:LPAC}.
\end{proof}

Note however that with Lemma~\ref{RedTMod} we are \emph{not} claiming that 
$M\sat$ ``$C$ is a model of $j(T)$'' implies that
\mbox{$V\sat$ ``$C_\Sigma$ is a model of $T$''}.  
For this further step we shall need $M$ to contain all of the 
tuples of length less than $\la$ of elements of $C$, 
so that the statement ``$C_\Sigma$ is a model of $T$'' is
absolute from $M$ to $V$.

\begin{lemma}
For any $\lambda<\ka$ and 
any theory $T$ for the language $\calL_\lambda(\Sigma)$,
if $A$ is a model of $T$, then $j(A)_\Sigma$ is a model of $T$.
\end{lemma}
\begin{proof}
This is immediate from Lemma~\ref{RedTMod} by elementarity.
\end{proof}

\begin{lemma}
Let $\cdot_\Sigma$ denote the function that takes $j(\Sigma)$-structures
$C$ to $C_\Sigma$ and takes $j(\Sigma)$-homomorphisms to $\Sigma$-homomorphisms
by leaving them unchanged on the underlying set of the domain.
Then
$\cdot_{\Sigma}$ is a functor from the category of $j(\Sigma)$-structures
to the category of $\Sigma$-structures.\hfill\qedsymbol
\end{lemma}

We shall studiously include $\Sigma$ subscripts in our notation for objects,
but omit them from homomorphisms, since they have no effect on them
as functions.

We denote by $\Str\Sigma$ the category of all $\Sigma$-structures with
homomorphism as the morphisms, and by 
$\Set$ the category of all sets with arbitrary functions as the morphisms.
In \cite{Ros:MDCM}, Rosick\'y allows for a change of language.  
We note that, much as Theorem~\ref{ModStr} seems to say that the theory is
irrelevant for $\ka$-directed colimits, so too is the language, in the
following sense.
For any $\ka$-ary signature $\Sigma$ and any $\ka$-directed diagram
in $\Str\Sigma$, the colimit of the diagram exists, and is the direct limit
of the structures.  In particular, the underlying set of the colimit is 
simply the colimit in $\Set$ of the diagram 
of the underlying sets, and the interpretations in the colimit 
of the relation and
function symbols of $\Sigma$
are then uniquely determined
(in the terminology of \cite{MacL:CWM}, the forgetful functor from
$\Str\Sigma$ to $\Set$, 
which takes each $\Sigma$-structure to its underlying set,
\emph{creates} $\ka$-directed colimits).
Here $\ka$-directedness is required so that every term
$f(\mathbf{a})$ in the colimit (for $f\in\Sigma$) 
appears in one of the structures of the
diagram --- one that contains every component of $\mathbf{a}$.
In fact $\Set$ is simply the special case in which we have reduced to empty
signature: we likewise have that 
for any extension $\ka$-ary signature $\Sigma'\supseteq\Sigma$,
the reduction functor from $\Str\Sigma'$ to $\Str\Sigma$ preserves 
$\lambda$-directed colimits.
In particular, to obtain \cite[Theorem~1]{Ros:MDCM}, it suffices to consider
a single language, which is the approach we take here.


We denote by $(j\restr\cdot)$ the natural transformation 
that associates to $A$ the
function $j\restr A:A\to j(A)_\Sigma$.
Of course, we should check that it is indeed a natural transformation.
\begin{lemma}\label{elemnat}
For any elementary embedding $j:V\to M$ with critical point $\ka$ 
and any $\la$-ary signature $\Sigma$ for some $\la<\ka$,
$(j\restr\cdot)$
is a natural transformation from the identity functor to the functor 
$\cdot_{\Sigma}\of j$.
That is, for any $\Sigma$-structure homomorphism $f:A\to B$, 
the following diagram commutes.
\[
\xymatrixcolsep{3pc}
\xymatrix{
A\ar[r]^{j\restr A}\ar[d]_f&j(A)_{\Sigma}\ar[d]^{j(f)}\\
B\ar[r]_{j\restr B}        &j(B)_{\Sigma}
}
\]
Moreover, if $M$ is closed under tuples of length less than $\la$, then
each morphism $j\restr A$ is an elementary embedding from 
$A$ to $j(A)_\Sigma$.
\end{lemma}
\begin{proof}
First note that for any $A$, $j\restr A$ is a $\Sigma$-structure homomorphism
to $j(A)_\Sigma$: if $A\sat R(\langle a_i\st i\in\alpha\rangle)$,
then by elementarity $j(A)\sat j(R)(\langle j(a_i)\st i\in\al\rangle)$
(in $M$, but this is absolute, so also in $V$),
so $j(A)_\Sigma\sat R(\langle j(a_i)\st i\in\al\rangle)$.
The corresponding statement holds for equations involving function symbols, 
showing that $j\restr A$
is a homomorphism. 
Further, if $M$ is closed under tuples of length less than $\la$, then
for \emph{any} first order formula $\phi$ in the language $\Sigma$,
$j(A)\sat\phi(\langle j(a_i)\st i\in\alpha\rangle)$ in $M$ if and only if
$j(A)\sat\phi(\langle j(a_i)\st i\in\alpha\rangle)$ in $V$, since
this satisfaction statement is $\Delta_1$ in the parameter $j(A)^{<\la}$,
and hence absolute between such $M$ and $V$.
Indeed, equivalent $\Sigma_1$ and $\Pi_1$ definitions may be extracted from the
usual recursive definition of $\sat$ for set-sized models, 
as given for example in 
\cite[Section~5.26]{AdR:LPAC}.
This shows that, entirely in $V$,
$j\restr A$ is an elementary embedding from $A$ to $j(A)$.

Now for any $a\in A$, $M\sat j(f)(j(a))= j(f(a))$, by elementarity.
But this statement is also absolute, and so also true in $V$.
\end{proof}

Penultimately for this section, we enunciate a simple observation that will be
useful.

\begin{lemma}\label{subcatcolimpres}
Suppose
$\calC_0$ is a full subcategory of $\calC_1$ and $\D$ is a diagram in
$\calC_0$.  If a colimit $C$ of $\D$ in $\calC_1$ exists and lies in
$\calC_0$, then $C$ is also a colimit of $\D$ in $\calC_0$, with the
same colimit cocone.
\end{lemma}

Finally, the following definition is useful
in the discussion both of infinitary languages and of large
cardinals.

\begin{defn}
For any set $X$ and any cardinal $\ka$, $\Power_\ka(X)$ denotes the the set of
all subsets of $X$ of cardinality less than $\ka$.
\end{defn}

\section{Colimits of structures and models}\label{StrModsection}

Theorem~\ref{ModStr} below is due to 
J. Rosick\'y {\cite[Theorem~1 and Remark~1~(2)]{Ros:MDCM}}. 
The proof given by  Rosick\'y uses 
atomic diagrams, explicit ultraproducts, 
and something called ``purity''; we avoid all that, using elementary embeddings
and ideas from the original paper of Richter \cite{Ric:LKR} that
Rosick\'y's theorem extends.  
We stick reasonably closely to Rosick\'y's notation in our proof of the theorem,
but note that we use $\D$ 
for the diagram rather than $D$ so that non-caligraphic uppercase Roman
letters near the start of the alphabet are always objects in one of 
our two main categories ($T$-models and $\Sigma$-structures).
Compositions of homomorphisms with cocones have the obvious meaning
(as do compositions on the other side of natural transformations with cocones).

The large cardinal axiom required for Theorem~\ref{ModStr} is the following.

\begin{defn}
If $\alpha \leq \kappa$ are uncountable cardinals, then we say that $\kappa$ is \emph{$\alpha$-strongly compact} if for every set $X$, every $\kappa$-complete filter on $X$ can be extended to an $\alpha$-complete ultrafilter on $X$.
The cardinal
$\kappa$ is \emph{strongly compact} if it is $\kappa$-strongly compact.
\end{defn}

Note that if $\ka\leq\la$ and $\ka$ is $\al$-strongly compact, then 
$\la$ is also $\al$-strongly compact.  In particular, $\al$-strongly
compact cardinals can be singular; but even further, 
the \emph{least} $\al$-strongly compact cardinal can be singular,
as was shown in \cite{Ba-Ma:GR} for $\al=\omega_1$ under the assumption of the
consistency of the existence of a supercompact cardinal.
We shall make use of the following characterization of 
$\al$-strongly compact cardinals in terms of elementary embeddings.

\begin{thm}[\mbox{\cite[Theorem~4.7]{Ba-Ma:GR}}]
\label{alphasc}
The following are equivalent for any uncountable cardinals $\alpha <\kappa$:
\begin{enumerate}
\item\label{BaMaGRstrcmpct}
$\kappa$ is $\alpha$-strongly compact.
\item\label{BaMaGRelem} 
For every $\gamma$ greater than or equal to $\kappa$ 
there exists an elementary embedding $j:V\to M$ definable in $V$, 
with $M$ transitive, 
${}^\al M\subset M$, $\crit(j)\geq\al$, and
such that 
there exists $Z\in M$ with 
$j`` \gamma = \{ j(\beta ):\beta <\gamma\}\subseteq Z$ 
and $M\models |Z|<j(\kappa)$. 
\item\label{BaMaGRfine} For every cardinal $\ga>\ka$,
there exists an $\alpha$-complete fine ultrafilter 
on $\mathcal{P}_\kappa (\ga)$.
\end{enumerate}
\end{thm}

Here a \emph{fine} ultrafilter $\calU$ on $\Power_\ka(\ga)$ is one such that 
for every $\al\in\ga$, 
\[
\{X\in\Power_\ka(\ga)\st \al\in X\}\in\calU.
\]
In the case of a regular $\kappa$, such an ultrafilter can be obtained using $\al$-strong compactness
by extending the $\ka$-complete filter generated by such sets.
The embedding $j$ in (\ref{BaMaGRelem}) 
is obtained  by  taking the ultrapower of $V$ 
by an $\alpha$-complete fine ultrafilter $\calU$ 
on $\Power_\ka(\ga)$ as in (\ref{BaMaGRfine}).

\begin{thm}\label{ModStr}
Let $\lambda$ be an infinite  cardinal,
$T$ a theory for the language $\calL_\lambda(\Sigma)$ over $\lambda$-ary signature
$\Sigma$,
and suppose there exists a cardinal $\kappa$ that is $\alpha$-strongly compact,
where $\alpha =\max\{ \lambda, \omega_1\}$. 
Suppose $\D$ is a $\kappa$-directed diagram of models for $T$, 
and suppose it has a colimit in the category $\Mod T$ of models of $T$.
The diagram $\D$ may also be considered to be a diagram in $\Str\Sigma$;
let $A$ be its colimit in this category.  Then $A$ is a model of $T$.
Hence, $A$ is the colimit of $\D$ in $\Mod T$.
\end{thm}
Note that the assumption that a $\Mod T$ colimit does exist is important ---
see Example~\ref{lowmax} below.
\begin{proof}
Since $\Mod T$ is a subcategory of $\Str\Sigma$, we will freely consider 
$\D$ to be a diagram in either as the context requires.
Let $\delta^A:\D\to A$ be the colimit cocone to $A$ as colimit of $\D$ in
$\Str\Sigma$.  Let $B$ denote the colimit of $\D$ in $\Mod T$, 
and let $\delta^B:\D\to B$ denote the colimit cocone.
Since $\delta^B$
is in particular a $\Sigma$-structure cocone, there is a unique
$\Sigma$-structure homomorphism $h$ from $A$ to $B$ such that
$h\of\delta^A=\delta^B$.  

The proof starts by chasing around the following 
diagram, in which bold font denotes diagrams, and
double-stemmed arrows ($\implies$) are used for
cocones, natural transformations generally, and
the inclusion $(j``\D)_\Sigma$ to $j(\D)_\Sigma$.  
\begin{equation}
\xymatrix{
&A\ar[rrrr]^{j\restr A}\ar[drr]^{g_A}\ar[dd]^h&&&&j(A)_\Sigma\ar[dd]_{j(h)}&\\
&&&\bar A_\Sigma\ar[urr]^{j(\delta^A)_{\bar A}}
\ar[drr]_{j(\delta^B)_{\bar A}}&&&\\
&B\ar[urr]_{g_B}\ar[rrrr]_(0.35){j\restr B}|-\hole&&&&j(B)_\Sigma\\
\boldsymbol{\D}\ar@{=>}[uuur]^{\delta^A}\ar@{=>}[ur]_{\delta^B}
\ar@{=>}[rrr]^{(j\restr\cdot)}&&&
\boldsymbol{(j``\D)_\Sigma}\ar@{=>}[uu]_(0.3)\zeta
\ar@{=>}[rrr]^{\text{inclusion}}&&&
\boldsymbol{j(\D)_\Sigma}\ar@{=>}[uuul]_{j(\delta^A)}
\ar@{=>}[ul]^{j(\delta^B)}\\
}\tag{$*$}
\end{equation}

First note that $h$ is epi for homomorphisms to $T$-models.
That is, if $f,g:B\to C$ are homomorphisms with codomain $C$ a 
$T$-model such that $f\circ h=g\circ h$, then $f=g$.  
For, considering cocones, we have
$f\circ\delta^B=f\circ h\circ\delta^A=g\circ h\circ\delta^A=g\circ\delta^B$,
from which the uniqueness-of-factorisation property of $B$ as a colimit
gives $f=g$.

Let $\ga$ be 
the number of objects in $\D$.
Let $j:V\to M$ be an ultrapower embedding for an $\al$-complete 
ultrafilter over $\Power_\ka(\ga)$,
as in (\ref{BaMaGRelem}) of Theorem~\ref{alphasc},
so that the critical point of $j$ is at least $\alpha$, 
$j$ is definable in $V$, and there exists $Z\in M$ such that $j`` \gamma =\{ j(\beta):\beta <\gamma\}\subseteq Z$ and $M\models |Z|<j(\kappa)$.  
Moreover, $M$ is closed under $\al$-tuples, so by the same argument as in
the proof of Lemma~\ref{elemnat}, being a model of $T$ is absolute between 
$M$ and $V$, and for every $\Sigma$-structure $C$, $j\restr C$ is 
elementary from $C$ to $j(C)_\Sigma$.

Since $\D$ is a $\ka$-directed diagram of models of $T$, 
\[
M\sat j(\D)\text{ is a }j(\ka)\text{-directed diagram of models of }j(T). 
\]
The objects in $j(\D)$ are $j(\Sigma)$-structures, which 
may be thought of as $\Sigma$-structures using $\cdot_{\Sigma}$.  
%
Since $Z\in M$, $j``\gamma \subseteq Z$, and $M\sat |Z|<j(\ka)$, 
we have that for any subset $X$ of $M$
of size at most $\gamma$, there is a $Y\in M$ such that $Y\supseteq X$ and
$M\sat |Y|<j(\ka)$ (see for example Kanamori~\cite[Theorem 22.4]{Kan:THI}).
In particular, since $j``\D$ has cardinality $\ga$ (the same as $\D$),
there is a $Y\in M$ with $j``\D\subseteq Y$ and $M\sat |Y|<j(\ka)$.
Intersecting such a $Y$ with $j(\D)$ and applying $j(\ka)$-directedness in $M$,
we conclude that
there is an object $\bar A$ in $j(\D)$ such that there are 
$j(\Sigma)$-homomorphisms to $\bar A$
from every object in $j``\D$, yielding a cocone $\zeta$ from the 
subdiagram $(j``\D)_{\Sigma}$ of $j(\D)_{\Sigma}$ to $\bar A_{\Sigma}$
(note: $\zeta$ as a whole is not necessarily in $M$, although each
component homomorphism of it is).
Using Lemma~\ref{elemnat},
these maps compose with the natural transformation $(j\restr\cdot)$
to give a cocone from $\D$ to $\bar A_{\Sigma}$ in $V$. 
Using the colimit definition of $A$, 
let $g_A:A\to\bar A_\Sigma$ be 
the unique $\Sigma$-structure homomorphism such that
$g_A\of\delta^A=\zeta\of(j\restr\cdot):\D\to\bar A_\Sigma$.
Moreover, since $\bar A$ is in $j(\D)$, $\bar A$ is a model of $j(T)$, 
so $\bar A_{\Sigma}$ is a model of $T$, and hence there is likewise 
a unique homomorphism $g_B:B\to\bar A_{\Sigma}$ of models of $T$ such that
$g_B\of\delta^B=\zeta\of(j\restr\cdot)$.
By the uniqueness of factorisation through $\delta^A$, we have
$g_B\of h=g_A$.
Also, there is a $j(\Sigma)$-structure map from $\bar A$ to $j(A)$,
namely the $\bar A$ component $j(\delta^A)_{\bar A}$
of the colimit cocone $j(\delta^A)$ from
$j(\D)$ to $j(A)$ in the category of $\Sigma$-structures of $M$ 
(of course, this is the colimit cocone in $M$ by
elementarity).
Likewise, we have the $\bar A$ component 
$j(\delta^B)_{\bar A}:\bar A\to j(B)$
of the colimit cocone $j(\delta^B)$ from
$j(\D)$ to $j(B)$ in the category of $T$ models of $M$.
Applying $\cdot_\Sigma$, we get $\Sigma$-structure
maps from $\bar A_\Sigma$ to $j(A)_\Sigma$ and $j(B)_\Sigma$.

There are two maps from $A$ to $j(A)_\Sigma$ that arise naturally:
$j\restr A$, and the map that exists because $A$ is the colimit of $\D$,
induced by the cocone $j(\delta^A)\circ(j\restr\cdot)$.  By uniqueness, the
latter map equals $j(\delta^A)_{\bar A}\circ g_A$.
By considering the concrete construction of the colimits $A$ and $j(A)$,
we see that these two maps are in fact the same: an element of $A$
given as $[a]$, the equivalence class of an element $a$ of some 
$D_i\in \D$, must be mapped in each case to the element 
$[j(a)]\in j(A)$.

Similarly, consider $j\restr B$, and the colimit map from $B$ to $j(B)_\Sigma$
induced by the cocone $j(\delta^B)\circ(j\restr\cdot)$,
which equals $j(\delta^B)_{\bar A}\circ g_B$.
In this case we cannot appeal to a concrete construction of $B$ to show that
they are the same.  However, their respective compositions with $h$ are both
equal to $j(h)\circ j\restr A$: by Lemma~\ref{elemnat} 
(that is, by elementarity of $j$)
in the case of $j\restr B\circ h$, and
by uniqueness of the map from $A$ to $j(B)_\Sigma$ factorising 
the relevant cocone
in the case of $j(\delta^B)_{\bar A}\circ g_B\circ h$.
Since $h$ is epi for homomorphisms to $T$-models as noted above, 
and $j(B)_\Sigma$ is a model of $T$, it follows that 
$j\restr B=j(\delta^B)_{\bar A}\circ g_B$.

We may therefore conclude that diagram ($*$) above commutes.  
Applying $j$ repeatedly, we now get a commutative diagram
\[
\xymatrix{
A\ar[rr]^{j\restr A}\ar[dr]^h&&j(A)_\Sigma\ar[rr]^{j(j\restr A)}\ar[dr]^{j(h)}&&
j^2(A)_\Sigma\ar[rr]^{j^2(j\restr A)}\ar[dr]^{j^2(h)}&&\cdots\\
&B\ar[rr]_{j\restr B}\ar[ur]&&
j(B)_\Sigma\ar[rr]_{j(j\restr B)}\ar[ur]&&
j^2(B)_\Sigma\ar[ur]\ar[r]
&\cdots\\
}
\]
with the horizontal
mappings all elementary embeddings.
The direct limits of the top and bottom horizontal chains 
therefore give structures into which $A$ and $B$ respectively
embed elementarily.  
But since the chains interleave, the direct limits must be the same.
Therefore, the (complete) theory of $A$ is the same as that of $B$,
and in particular, $A$ is a model of $T$.
\end{proof}

The following example shows that the assumption that there is a colimit in
$\Mod T$ was necessary.

\begin{eg}\label{lowmax}
Let $\Sigma=\{<\}$, where as usual $<$ is a binary relation, 
and let $T$ be the theory of linear orders with a
maximum element.  Let $\D$ be the diagram whose objects are
all ordinals less than $\ka$, and whose morphisms are the usual inclusions.
Then $\D$ is $\ka$-directed, and its colimit in $\Str\Sigma$ is $\ka$,
but $\D$ has no colimit in $\Mod T$.  
\end{eg}
One might wonder if $\ka+1$ could
be a colimit for $\D$ in $\Mod T$ in Example~\ref{lowmax},
but since we are not naming the maximum element with a constant, the 
uniqueness property of colimits rules this out: for example, there are two
order preserving functions from $\ka+1$ to $\ka+2$ preserving the
inclusions of the ordinals less than $\ka$.

\section{Colimits of more general categories}\label{GenCatssection}

An important general notion in category theory encompassing many
categories of interest is that of an
\emph{accessible category}.
An accessible category is a category  that, for some regular cardinal
$\la$, has $\la$-directed colimits and a \emph{set} of nice
(specifically, $\la$-presentable) objects which generate the category
by $\la$-directed colimits --- see \cite[Chapter~2]{AdR:LPAC} for precise 
details.

A large cardinal axiom that has found great applicability in the 
study of accessible categories (see for example \cite[Chapter~6]{AdR:LPAC})
is the following.

\begin{defn}
\Vopenka's Principle (VP) is the statement that for every signature $\Sigma$
and every proper class $\calC$ of $\Sigma$-structures, there are two members
$A$ and $B$ of $\calC$ such that there exists a (non-identity) 
elementary embedding $j:A\to B$.
\end{defn}

Note that quantifying over proper classes is not permitted in standard 
ZFC set theory, so we treat VP as an axiom schema,
giving an axiom for each formula that defines a proper class.

In \cite{RTA:UPLPC} (see also \cite[Theorem~6.9]{AdR:LPAC}), 
Rosick\'y, Trnkov\'a and Ad\'amek 
prove the following colimit preservation theorem
for accessible categories.

\begin{thm}\label{AdRThm6.9}
Assuming Vop\v{e}nka's Principle, 
for each full embedding $F:\mathcal{A}\to \mathcal{K}$, where $\mathcal{K}$ is an accessible category, there exists a regular cardinal $\lambda$ such that $F$ preserves $\lambda$-directed colimits.
\end{thm}

The conclusion of Theorem~\ref{AdRThm6.9} is in fact equivalent to VP 
as shown by Example~6.12 of \cite{AdR:LPAC}.

In Theorem~\ref{thm2}
we prove a result that is simultaneously more general and
more refined, using appropriate  fragments of VP. 
Specifically, we use  \emph{$C^{(n)}$-extendible cardinals}.

We recall the \emph{L\'evy hierarchy} of formulas.
A formula is said to be $\Sigma_0$, $\Pi_0$, or $\Delta_0$ if it involves no
unbounded quantifiers.  
For $n>0$, a formula (and the notion it expresses)
is said to be $\Sigma_n$ if it is of the form $\exists x(\phi(x))$ for
some $\phi\in\Pi_{n-1}$, and $\Pi_n$ if it is of the form 
$\forall x(\phi(x))$ for some $\phi\in\Sigma_{n-1}$.  A notion is said
to be $\Delta_n$ if it can be expressed equivalently by a formula that
is $\Sigma_n$ or a formula that is $\Pi_n$.  
\label{Delta1Defn}
For example, if ZFC implies
that there will be a unique $x$ with some property $\phi$, then
$\exists x(\phi(x)\land\psi(x))$ and $\forall x(\phi(x)\implies\psi(x))$
will be equivalent.

Following \cite{Ba:CC}, we denote by $C^{(n)}$
the closed and unbounded proper class of cardinals $\lambda$ 
such that $V_\lambda$ is a $\Sigma_n$-elementary substructure of $V$,
that is, a $\Sigma_n$ statement is true in $V_\la$ if and only if
it is true in $V$. 
A cardinal $\kappa$ is called \emph{$C^{(n)}$-extendible} if for every 
$\lambda>\kappa$,  
there is an elementary embedding $j:V_\lambda \to V_\mu$ for some $\mu>\la$,
with critical point $\kappa$ and with $j(\kappa)$ a cardinal in $C^{(n)}$ 
greater than $\la$.
A natural strengthening of this is also considered: following
\cite{Ba:CC}, we say that a cardinal is \emph{$C^{(n)+}$-extendible} 
if for every $\la>\ka$ in $C^{(n)}$ there is is an elementary embedding
$j:V_\la\to V_\mu$ for some $\mu>\la$ in $C^{(n)}$, with $\crit(j)=\ka$
and $j(\ka)>\la$ in $C^{(n)}$ 
(this was actually the definition of $C^{(n)}$-extendible cardinal used in
\cite{BCMR:DOCACS}).
Note that every $C^{(n)+}$-extendible cardinal $\ka$ 
is $C^{(n)}$-extendible, as for
every ordinal $\xi>\ka$ there is a cardinal 
$\la$ greater than $\xi$ in $C^{(n)}$,
and if $j:V_\la\to V_\mu$ witnesses the $\la$-$C^{(n)+}$-extendibility of 
$\ka$, then $j\restr V_\xi$ witnesses the $\xi$-$C^{(n)}$-extendibility of 
$\ka$.
The two definitions are in fact even more closely related.

\begin{prop}\label{CnExtAboveEquiv}
For every cardinal $\al$, the statement 
``there is a $C^{(n)}$-extendible cardinal greater than $\al$''
is equivalent to
``there is a $C^{(n)+}$-extendible cardinal greater than $\al$''
\end{prop}
\begin{proof}
The proof is evident from a careful reading of \cite{Ba:CC}.
Theorem~4.11 of \cite{Ba:CC} shows that if $\ka$ is $C^{(n)}$-extendible then
\Vopenka's Principle holds for $\Sigma_{n+2}$-definable proper classes with 
parameters in $V_\ka$.  The proof of Theorem~4.12 of \cite{Ba:CC}, modified
as described in the remarks that follow it in that paper, then shows that
for every $\xi<\ka$, there is a $C^{(n)+}$-extendible cardinal greater than 
$\xi$.  Thus, ``$\ka$ is $C^{(n)}$-extendible'' implies that
``for all $\xi<\ka$, there is a $C^{(n)+}$-extendible cardinal greater than 
$\xi$'', which is clearly equivalent to the substantive direction of the 
Proposition.
\end{proof}

We can in fact strengthen this further.

\begin{prop}
If a cardinal is $C{(n)}$-extendible, then it is either 
$C^{(n)+}$-extendible, or it is a limit of $C^{(n)+}$-extendible cardinals.
\end{prop}
\begin{proof}
Suppose $\ka$ is $C^{(n)}$-extendible but not the limit of
$C^{(n)+}$-extendible cardinals.  Then there is some $\xi<\ka$ such that
for all $\zeta$ strictly between $\xi$ and $\ka$, $\zeta$ is not
$C^{(n)+}$-extendible.  By Proposition~\ref{CnExtAboveEquiv}, there is
a $C^{(n)+}$-extendible cardinal $\la$ greater than $\xi$, and hence greater
than or equal to $\ka$.  If $\la=\ka$ we are done, so suppose $\la>\ka$.
We may assume without loss of generality that $\la$ is the \emph{least}
$C^{(n)+}$-extendible cardinal strictly greater than $\ka$.
Since $\la$ is $C^{(n)}$-extendible it is in $C^{(n+2)}$
(\cite[Proposition~3.4]{Ba:CC}), and so since 
``$\ka$ is $C^{(n)}$-extendible'' and
``$\ka$ is $C^{(n)+}$-extendible'' 
are $\Pi_{n+2}$ statements
(again, see \cite{Ba:CC}), we have that
\[
V_\la\sat\ka\text{ is }C^{(n)}\text{-extendible}
\land\forall\zeta
\left((\zeta>\xi\land\zeta\neq\ka)\implies\zeta
\text{ is not }C^{(n)+}\text{-extendible}\right).
\]
But now $\la$ is also inaccessible, so full ZFC holds in $V_\la$, and in 
particular Proposition~\ref{CnExtAboveEquiv}.
Thus we may deduce that that 
$V_\la\sat\ka\text{ is }C^{(n)+}\text{-extendible}$,
whence by $\Sigma_{n+2}$-correctness of $V_\la$ again we have
that $\ka$ is $C^{(n)+}$-extendible.
\end{proof}

This of course raises a natural question.

\begin{open}  Is it consistent to have a $C^{(n)}$-extendible cardinal that
is not $C^{(n)+}$-extendible?
\end{open}

It is shown in \cite{Ba:CC} that VP is equivalent to the existence of a proper class of $C^{(n)}$-extendible cardinals for every $n$.
Moreover this is a precise stratification, with the existence of 
a $C^{(n)}$-extendible cardinal $\ka$ corresponding to 
VP for $\Sigma_{n+2}$-definable classes with parameters in $V_\ka$.
We now show that this same stratification is applicable to
Theorem~\ref{AdRThm6.9}.
We also extend the scope of the theorem to a wider range of categories $\calK$,
noting that every accessible category may be embedded as a full subcategory
of $\Str\Sigma$ 
(see for example \cite[Characterization Theorem~5.35]{AdR:LPAC}).

We use the convention of \cite{BCMR:DOCACS}, calling a category $\calC$
$\Sigma_n$ definable if there is a $\Sigma_n$ formula $\phi(x_1,\ldots,x_8)$
(in the language of set theory) 
and a parameter $p$ such that
$\phi(A,B,C,f,g,h,i,p)$
holds if and only if
$A, B$ and $C$ are objects of $\calC$,
$f\in\calC(A,B)$,
$g\in\calC(B,C)$,
$h\in\calC(A,C)$,
$h=g\of f$ and $i=\Id_A$.
Note that from such a $\phi$, formulas may be obtained for $\Obj(\calC)$,
$\Mor(\calC)$, $\circ$ and $\Id$, with the only extra quantification an
$\exists i$ for some of them.
We say that a functor $F:\calC_0\to \calC_1$ is $\Sigma_n$ definable 
if there are
$\Sigma_n$ formulas $\phi^F_{Obj}(x,y)$ and $\phi^F_{Mor}(x,y)$ such that 
for any object $A$ and morphism $f$ of $\calC_0$,
$\phi^F_{Obj}(A,B)$ holds if and only if $B=F(A)$, and
$\phi^F_{Mor}(f,g)$ holds if and only if $g=F(f)$.

An annoying quirk of using infinitary languages is that $\Str\Sigma$ need not
be absolute: for example, a function defined on all countably infinite tuples 
from a set $X$ in some set theoretic universe will
no longer be defined on all countably infinite tuples of $X$
if we move to a universe
with more countably infinite subsets of $X$.   
However, this obstacle, generalised 
to arbitrary infinite cardinalities, is the only obstruction to absoluteness.
\begin{prop}
Let $\Sigma$ be a signature.  If $\Sigma$ contains no infinitary 
function symbols, then $X\in\Obj(\Str\Sigma)$ is $\Delta_1$ definable
with $\Sigma$ as a parameter; 
otherwise, $X\in\Obj(\Str\Sigma)$ is $\Pi_1$ definable with parameter $\Sigma$.
Moreover, if for some $\ka$ greater than the arities of all the function
symbols in $\Sigma$, we add the function $\Power_\ka$ to the language
of set theory (where $\Power_\ka(X)$ denotes the set of all subsets of $X$ of
cardinality less than $\ka$), 
then $X\in\Obj(\Str\Sigma)$ is $\Delta_0$ definable for this
extended language (again with $\Sigma$ as a parameter).
\end{prop}
The point is that $\Delta_1$ definability implies absoluteness between
models of set theory,
whereas $\Pi_1$ definability only implies downward absoluteness.
The second part of the Proposition tells us that we have absoluteness
of $\Obj(\Str\Sigma)$
between models of set theory that agree about the function $\Power_\ka$.
\begin{proof}
With the precise definition of the notion of a $\Sigma$-structure from
Section~\ref{preliminaries},
it is straightforward to show that
$A\in\Obj(\Str\Sigma)$
is $\Delta_0$ in the parameters $\Sigma$ and $\Power_\ka(A)$, where
$\ka$ is greater than all of the arities of symbols in $\Sigma$.
It is well known that $\Power_{\aleph_0}(A)$ is $\Delta_1$ definable from $A$,
but for greater $\ka$ it is $\Pi_1$ definable with parameter $\ka$; 
since it is unique, adding
its definition to the formula makes the expression $\Delta_1$ in the first case
(see the comment on page \pageref{Delta1Defn} where we defined $\Delta_1$)
and $\Pi_1$ in the second.  Moreover,
for each $\la$-ary relation symbol $R$, 
one just needs to verify that each element of $R^{A}$ is a function
from $\la$ to $A$,
and so rather than $\Power_{\la^+}(A)$ as a parameter
it suffices for the definition to just have $\la$, which is
recoverable from $R$ itself.
\end{proof}

Note that the initial segments $V_\la$ of $V$, which are relevant to
$C^{(n)}$-extendible and $C^{(n)+}$-extendible cardinals, 
are \emph{correct} for $\Power_\ka$:
for any $X$ in $V_\la$ for $\la$ a limit ordinal, $V_\la$ contains every
subset of $X$ of cardinality less than $\ka$ (and indeed, every subset of $X$ 
of any cardinality), and so $V_\la$ agrees with $V$ about the function 
$\Power_\ka$.  Thus, a $C^{(n)}$-extendible cardinal $\ka$ has
embeddings witnessing its $C^{(n)}$-extendibility that are
actually elementary for formulas in the language of set theory extended by 
$\Power_\ka$ (and likewise for $C^{(n)+}$-extendibility).

We claimed above that Theorem~\ref{thm2} is more general than
Theorem~\ref{AdRThm6.9}.  Certainly in a ZFC setting every category 
(indeed every class) is definable, 
and so $\Sigma_n$-definable for some $n$.
In Theorem~\ref{thm2} we require $\calK$ to be a full subcategory of 
$\Str\Sigma$ for some signature $\Sigma$, but it turns out that this still
allows a vast array of categories, including all accessible ones:
we have the following characterisation of accessible categories
(see \cite[Theorem~5.35]{AdR:LPAC}).

\begin{thm}
Accessible categories are precisely the categories equivalent to categories of 
models of \emph{basic} theories, that is, those whose formulas are
universally quantified implications of positive-existential formulas.
\end{thm}

\begin{thm}
\label{thm2}
Suppose that $n>0$, $\mathcal{K}$ is a 
full subcategory of $\Str\Sigma$ for some 
signature $\Sigma$, and $F:\mathcal{A}\to \mathcal{K}$ is a 
$\Sigma_n$ definable full embedding with $\Sigma_n$ definable domain
category $\calA$. 
If there exists  a $C^{(n)}$-extendible cardinal 
greater than the rank of $\Sigma$, 
the arity of each function or relation symbol in $\Sigma$, 
and the rank of the parameters involved in some $\Sigma_n$ definitions of  
$F$ and $\calA$ and some definition of $\calK$, 
then there exists a regular cardinal $\lambda$ such that $F$ preserves $\lambda$-directed colimits.
\end{thm}

\begin{proof}
By Lemma~\ref{subcatcolimpres} it suffices to show that the embedding
$i\of F:\calA\to\Str\Sigma$ preserves $\lambda$-directed colimits for some
regular cardinal $\la$, where $i$ is the inclusion functor
from $\calK$ to $\Str\Sigma$. 
Note that if $F$ is $\Sigma_n$ definable as a functor to $\calK$ then
it remains so as a functor to $\Str\Sigma$, that is, 
$i\of F$ is $\Sigma_n$ definable.
Thus, let us assume without loss of generality that $\calK=\Str\Sigma$.
In particular, we may use the fact that
$\Str\Sigma$ has all $\la$-directed colimits for
$\la$ greater than all of the arities of symbols in $\Sigma$.

For terminological convenience let $\beta$ be a cardinal greater than  
the rank of $\Sigma$, the arity of each of the symbols in $\Sigma$, and the
rank of the parameters involved in the definitions of $F$ and $\calA$.
Let  $\mathcal{C}$ be the category  whose objects are maps $a:\bar{A}\to F(A)$,
where, for some regular cardinal $\lambda>\beta$
and some  $\lambda$-directed diagram $\D$ in $\mathcal{A}$, 
$A$ is a colimit in $\mathcal{A}$ of $\D$, 
$\bar{A}$ is a colimit of $F\D$ in $\K$, and $a$ is the homomorphism
induced by the image of the $\calA$-colimit cocone under $F$.
The morphisms $f:a \to b$ are pairs $f=\langle g,h\rangle$, with $g\in Hom_{\K}(\bar{A} , \bar{B})$ and $h\in Hom_{\K}(F(A) ,F(B))$,  such that the following diagram commutes:
\[
\xymatrixcolsep{3pc}
\xymatrix{
\bar{A} \ar[r]^{a}\ar[d]_g&F(A)\ar[d]^{h}\\
\bar{B}\ar[r]_{b}        &F(B).}
\]
Let $\mathcal{C}^\ast$ be the full subcategory of $\mathcal{C}$ 
whose objects are those $a\in\Obj(\mathcal{C})$ that are not isomorphisms. 
If  the conclusion of the theorem fails, that is, if for every regular cardinal $\lambda$, some $\lambda$-directed colimit is not preserved by $F$, 
then  
even up to isomorphism $\calC^*$ has a proper class of objects:
in the terminology of \cite[Section~0.1]{AdR:LPAC}, $\calC^*$ is not 
essentially small.

We claim that membership in $\Obj({\calC^*})$ is $\Sigma_{n+2}$ definable
over the language of set theory extended by $\Power_\beta$. 
We have: $a\in\Obj({\mathcal{C}^\ast})$ if and only if
\begin{align*}
\exists \lambda \exists \D\exists\langle\bar{A},\bar{\eta}\rangle
\exists\langle A,\eta\rangle(&
\lambda\mbox{ is a regular cardinal }\land
\D\text{ is a diagram in }\mathcal{A}\,\land\\
&\D\text{ is $\la$-directed }\land\\
&\langle\bar{A},\bar\eta\rangle=\Colim_{\K}(F\D) \wedge
\langle A,\eta\rangle=\Colim_{\mathcal{A}}(\D)\,  \wedge\\
&a:\bar{A}\to F(A)\mbox{ is the induced homomorphism }\wedge\\
&a\mbox{ is not an isomorphism}).
\end{align*}
Here we are treating a diagram $\D$ as a set of objects and morphisms,
and we use $\langle A,\eta\rangle=\Colim_\calA(\D)$ to mean that
$A$ is the colimit in $\calA$ of $\D$ with colimit cocone $\eta$.
The statement
``$a$ is not an isomorphism'' is $\Delta_0$ in $\bar A$ and $A$,
``$\la$ is a regular cardinal'' and
``$\D$ is $\la$-directed'' are $\Pi_1$,
and
``$\D$ is a diagram in $\calA$'' is $\Sigma_n$ since $\calA$ is.
The statement that ``$a:\bar{A}\to F(A)$ is the induced homomorphism''
can simply be expressed by saying that $a$ is a homomorphism from $\bar A$ to
$F(A)$ and the requisite triangles with cocone maps (indexed by objects of $\D$)
commute, so this part of the formula is just $\Sigma_n$ because $F$ needs to 
be evaluated for it.
In terms of quantifier complexity the crux is really the statement
$\langle A,\eta\rangle=\Colim_\calA(\D)$,
as this is equivalent to saying that for every cocone over $\D$ in $\calA$ there
is a unique morphism in $\calA$ from $A$ to the vertex object of that cocone
making everything relevant commute; this can be expressed by a $\Pi_{n+1}$
formula, handling existence and uniqueness separately to save on quantifiers. 
Similarly,
$\langle\bar A,\bar \eta\rangle=\Colim_{\Str\Sigma}(F\D)$ is 
$\Pi_{n+1}$.

Assume for the sake of obtaining a contradiction 
that $\calC^*$ is not essentially small.
Let $\ka$ be a $C^{(n)}$-extendible cardinal greater than $\beta$
so that embeddings with critical point $\ka$ do not affect 
the parameters in 
the definitions of $\K$, $\calA$ or $F$.
By Proposition~\ref{CnExtAboveEquiv}, we may assume without loss of generality
that $\ka$ is in fact $C^{(n)+}$-extendible.
Let $a$ be an object of $\calC^*$ of rank greater than $\ka$,
arising from a $\la_a$-directed diagram $\D_a$ 
for some $\la_a$ also greater than 
$\ka$.
Let $\lambda\in C^{(n)}$ be greater than 
the ranks of $a$, $\D_a$, $F\D_a$, and the corresponding colimit cocones
$\langle\bar A,\bar\eta\rangle_a$ and 
$\langle A,\eta\rangle_a$;
in particular, sufficiently large that
$V_\la$ contains witnesses to the fact that $a\in\Obj({\calC^*})$.
Let $j:V_\lambda \to V_\mu$ be an elementary embedding with critical point 
$\kappa$,  such that $\mu>j(\kappa)>\lambda$ are all in $C^{(n)}$. 
In particular, 
\[
V_\mu\sat ``\la_a,\D_a, \langle\bar A,\bar\eta\rangle_a,\text{ and }
\langle A,\eta\rangle_a\text{ witness that }a\in\Obj({\mathcal{C}^\ast})", 
\]
since the statement in quotes is $\Pi_{n+1}$, and hence downwardly absolute
from $V$ to $\Sigma_{n}$-correct sets such as $V_\mu$ which contain all of the
parameters.
Of course,
not everything that $V_\mu$ believes to be in $\Obj({\mathcal{C}^\ast})$
need be so in $V$.
However, it will suffice for our purposes to carry out the remainder of the
argument within $V_\mu$, obtaining a contradiction from the fact that 
$a$ is not an isomorphism (whether construed in $V_\mu$ or $V$).
We use the standard notation of using a superscript $V_\mu$ to indicate that
an expression is to be interpreted in $V_\mu$; thus for example
${\calC^\ast}^{V_\mu}$ denotes the set
$\{x\in V_\mu\st V_\mu\sat x\in\calC^{\ast}\}$.
Also note that, as they have $\Sigma_n$ definitions, 
membership in $\calA$ and the
evaluation of $F$ are in any case absolute to $V_\mu$.

By the choice of $\ka>\beta$, we have that $j$ is the identity on the
parameters to the defintion of $F$, and so $j$ commutes with $F$.
Since $j$ is elementary, 
we have a morphism of ${\mathcal{C}^\ast}^{V_\mu}$ 
\[
\xymatrix{
\bar{A} \ar[r]^{a}\ar[d]_{j\restr\bar{A}}&F(A)\ar[d]^{j\restr F(A)}\\
j(\bar{A})\ar[r]_{j(a)}        &j(F(A)),}
\]
which we denote by $j\restriction a: a \to j(a)$.
Indeed, by elementarity $j(a)$ is the induced homomorphism from $j(\bar A)$ 
(the colimit in $\K$ of $j(F\D_a)$) to
$j(F(A))=F(j(A))$.
Further, $j(a)$ is not an isomorphism, and so lies in ${\calC^*}^{V_\mu}$.
The commutativity of the diagram also follows by elementarity, as in
Lemma~\ref{elemnat}:
for any element $\al$ of the $\Sigma$-structure $\bar A$,
$j(a)\of j\restr \bar A\,(\al)=j(a)(j(\al))=j(a(\al))=(j\restr F(A)\of a)(\al)$.

Since $j(\la_a)>j(\ka)>\la$ and $j(\D_a)$ is $j(\la_a)$-directed, 
$j(F\D_a)=Fj(\D_a)$ is certainly
$\la$-directed.  
The set of objects $j``\Obj(F\D_a)$ is a subset of cardinality less than $\la$ 
of the objects of $Fj(\D_a)$, and hence has an upper bound $F(d_0)$ in
$Fj(\D_a)$.  Composing the cocone from $j``F\D_a$ to $F(d_0)$ with the
natural transformation $j\restr\cdot$ from $F\D_a$ to $j``(F\D_a)$ 
(see Lemma~\ref{elemnat}), we get a cocone from $F\D_a$ to $F(d_0)$.
The picture is essentially the same as shown in diagram ($*$) in the proof of 
Theorem~\ref{ModStr}.
\begin{equation*}
\xymatrix{
&\bar{A}\ar[rrrr]^{j\restr\bar{A}}\ar[drr]^{g_{\bar A}}\ar[dd]^a&&&&j(\bar{A})\ar[dd]_{j(a)}&\\
&&&F(d_0)\ar[urr]^{j(\delta^{\bar{A}})_{F(d_0)}}
\ar[drr]_{F(j(\delta^A)_{d_0})}&&&\\
&F(A)\ar[urr]_{F(g_A)}\ar[rrrr]_(0.35){j\restr F(A)}|-\hole&&&&F(j(A))\\
\boldsymbol{F\D_a}\ar@{=>}[uuur]^{\delta^{\bar{A}}}\ar@{=>}[ur]_{F\delta^{A}}
\ar@{=>}[rrr]^{(j\restr\cdot)}&&&
\boldsymbol{F(j``\D_a)}\ar@{=>}[uu]_(0.3)\zeta
\ar@{=>}[rrr]^{\text{inclusion}}&&&
\boldsymbol{F(j(\D_a))}\ar@{=>}[uuul]_{j(\delta^{\bar{A}})}
\ar@{=>}[ul]^{Fj(\delta^{A})}\\
}
\end{equation*}

As in the proof of Theorem~\ref{ModStr}, we let 
$g_{\bar A}:\bar A\to F(d_0)$ and $F(g_A):FA\to F(d_0)$ be 
the maps induced by the cocone from $F\D$ to $F(d_0)$.
Since the maps $j\restr\bar A:\bar A\to j(\bar A)$ and
$j\restr F(A):F(A)\to F(j(A))$ make the diagrams with the corresponding 
cocones commute, they must be the induced colimit maps.

We may now deduce that $a$ is an isomorphism.  By elementarity once again, the
homomorphism $j\restr\bar A:\bar{A}\to j(\bar{A})$ is injective 
and preserves the complements of the relations in $\Sigma$;
thus, since it factors through $a$
as $j\restr\bar A=j(\delta^{\bar A})_{F(d_0)}\of F(g_A)\of a$, 
the same is true of $a$.  It therefore 
only remains to show that $a$ is surjective.  If we had 
$\al\in F(A)\smallsetminus a``\bar{A}$, then by elementarity
$j(\al)$ would lie in $F(j(A))\smallsetminus j(a)``j(\bar A)$, 
contradicting the fact that 
$j(\al)=j(a)\of\delta^{j(\bar A)}_{F(d_0)}\of F(g_A)(\al)$.
Hence, $a$ is indeed an isomorphism in $V_\mu$.
But this contradicts the definition of ${\calC^*}^{V_\mu}$, and 
so returning to our initial assumption, we may conclude
that (in $V$) $\calC^*$ is essentially small, as required.
\end{proof}

In light of Theorem~\ref{thm2}, it is natural to ask whether 
Theorem~\ref{ModStr} can be generalised.  Central to the proof of 
Theorem~\ref{ModStr}
was the absoluteness of $\Mod T$ between different models of set theory
(again, with the $\Power_\ka$ function added to the language of set theory
for $\ka$ greater than the arities of the symbols in $\Sigma$).
This absoluteness arises from the fact that $\Mod T$ is 
$\Delta_1$ definable in the parameters $\Sigma$ and $T$, 
and it turns out that for such categories a generalisation is indeed possible,
using the same large cardinal assumption far below $C^{(n)}$-extendible
cardinals in strength.

\begin{thm}\label{Delta1}
Suppose that $\calK$ is a definable full subcategory of $\Str\Sigma$ for some
signature $\Sigma$, and $\calA$ is a $\Delta_1$ definable full subcategory of
$\calK$.
If there exists an $\alpha$-strongly compact cardinal $\ka$
for some $\al$ greater than the ranks of the parameters in a $\Delta_1$ 
definition of $\calA$ and a definition of $\calK$,
then the inclusion functor $\calA\into\calK$ preserves $\ka$-directed colimits.
\end{thm}
\begin{proof}
The proof exactly follows that of Theorem~\ref{ModStr} to build up the
analogue of diagram ($*$), using the absoluteness of $\Delta_1$
definitions between models of set theory.  
From that point, whilst we cannot use an elementary chains argument in this
context, the argument from the proof of Theorem~\ref{thm2}
that the morphism from the $\calK$-colimit to the $\calA$-colimit is an
isomorphism \emph{does} translate.
\end{proof}

As already mentioned above, the conclusion of Theorem~\ref{AdRThm6.9} is in fact equivalent to VP 
as shown by Example~6.12 of \cite{AdR:LPAC}.
The example may also be stratified, to show that some degree of 
large cardinal strength is necessary for the conclusion of 
Theorem~\ref{thm2}.

\begin{thm}\label{reversal}
For any $n\geq1$,
suppose that for every signature $\Sigma$, every $\bPi{n+1}$ definable full
subcategory $\calK$ of $\Str\Sigma$, 
and every $\bPi{n+1}$ definable full embedding
$F:\calA\to\calK$, there is a regular cardinal $\la$ such that $F$ preserves
$\la$-directed colimits.  Then there exists a proper class of 
$C^{(n)}$-extendible cardinals.
\end{thm}
\begin{proof}
We prove the contrapositive by means of a counterexample.
In \cite{Ba:CC} it is shown that the existence of a proper class of 
$C^{(n)}$-extendible cardinals is equivalent to \Vopenka's Principle for 
$\Sigma_{n+2}$ classes of structures.
Indeed, the proof of Theorem 4.12 of \cite{Ba:CC} exhibits,
under the assumption that there are no $C^{(n)}$-extendible cardinals,
a $\Pi_{n+1}$ class of structures, between distinct elements of which 
there can be no elementary 
embeddings.  If there is some bound $\beta$ such that
all $C^{(n)}$-extendible cardinals are less than $\beta$, then
the same construction can be employed starting from $\beta$ to again
give a $\Pi_{n+1}$ class with no elementary embeddings between distinct 
elements.  Changing the construction slightly to give each structure an
underlying set of the form $V_{\la_\al+2}$ further eliminates the possibility
of non-trivial elementary embeddings from one of the structures to itself
by Kunen's inconsistency theorem \cite{Kun:ElEm}, without changing the
complexity of the definition.
Another useful feature of the construction is that the structures are over
a finitary relational signature $\Sigma$.

So suppose we have 
such a $\Pi_{n+1}$ class $\calC$ of $\Sigma$-structures with no
elementary embeddings for some finitary relational $\Sigma$.
We expand $\Sigma$ to a signature $\Sigma'$ by adding
\emph{Skolem relations}. Using G\"odel numbering, we may take a 
recursive bijection $i\mapsto\phi_i$ 
between natural numbers and $\Sigma$-formulas.
Moreover, this may be done in such a way that $\phi_i$ has at most $i$
free variables, and by ignoring extra values we may treat $\phi_i$ as
having exactly $i$ free variables.
We take $\Sigma'=\Sigma\cup\{R_i\st i\in\omega\}$, where for each $i$,
$R_i$ is an $i$-ary relation symbol.
For each structure $M\in\calC$ we take an expanded structure
$M'$ given by $M$ concatenated with $\langle R_i^{M}\st i\in\omega\rangle$,
where for each $i$ and $M$, we set
\[
R_i^M=\{(m_1,\ldots,m_i)\in M^i\st M\sat\phi_i(m_1,\ldots,m_n)\}.
\]
The satisfaction relation $\sat$ is known to be $\Delta_1$ for finitary 
languages, and so the class $\calC'=\{M'\st M\in\calC\}$ remains $\Pi_{n+1}$.
Moreover, 
$\calC'$ admits no non-trivial ($\Sigma'$-) homomorphisms between its elements,
as a homomorphism $h:M'\to N'$ must restrict to an elementary embedding
$M\to N$.

With such a class $\calC'$ to hand, the argument now proceeds very much as
for \cite[Example~6.12]{AdR:LPAC}. 
Let $\calA$ be the full subcategory of $\Sigma'$-structures
consisting 
of those $\Sigma'$-structures $A$ such that 
$\Hom(C,A)$ is empty for all $C\in\calC'$, as well as
the terminal object $T$
(a single point $\Sigma'$-structure, 
with $R^T=T^i$ for each $i$-ary relation $R$).
The only unbounded quantifiers this definition uses to build up from
that of $\calC'$ are universal,
so $\calA$ is also $\Pi_{n+1}$, and the inclusion of $\calA$ into $\Str\Sigma'$
is therefore also $\Pi_{n+1}$.
For each regular cardinal $\la$, consider $C\in\calC'$ of cardinality at 
least $\la$. Clearly each
proper substructure of $C$ lies in $\calA$,
so we may consider the $\la$-directed diagram $\D$ of all
substructures of $C$ of cardinality strictly less than $\la$
(since $\Sigma'$ is relational, these are just the $<\la$-sized subsets
with the induced structure).
In $\Str\Sigma'$, the colimit of $\D$ is $C$, so by the definition
of $\calA$, the only cocone on $\D$ in $\calA$ is the cocone to the
terminal object.  Thus, $\Colim_\calA(\D)=T\neq C$, so the
inclusion functor does not preserve $\la$-directed colimits.
\end{proof}

The complexities in Theorems~\ref{thm2} and \ref{reversal} are such that
we lose strength moving from large cardinals to colimit preservation
and back again.  It would be very interesting to improve one or
both of these Theorems to close this gap.

One might also hope to draw large cardinal strength at the bottom of
the definability hierarchy from the equivalence shown in 
\cite{BCMR:DOCACS} between \Vopenka's Principle
for ${\Delta_2}$ classes and the existence of a 
proper class of supercompact cardinals.
However, a na\"{\i}ve modification of the proof of Theorem~\ref{reversal}
is fruitless, 
as the extra universal quantifier involved in the definition of $\calA$
would take us up to 
${\Pi_2}$, for which the $n=1$ case of Theorem~\ref{reversal}
is already a better result.

However, we now show that one \emph{can} obtain large cardinal strength
from a ${\Delta_2}$ example, which moreover is a less contrived
example than those used to prove Theorem~\ref{reversal}.
For this we use the notion of a \emph{group radical}.

\begin{defn}\label{radicals}
For any abelian group $X$,
the \emph{radical singly generated by $X$} is the functor
$R_X$ from abelian groups to abelian groups given by
\[
R_X(G)=\bigcap_{f\in\Hom(G,X)}\ker(f),
\]
with the action on homomorphisms simply given by restricting.
For any cardinal $\ka$, we define the functor $R_X^\ka$ by
\[
R_X^\ka(G)=\sum_{A\leq G, |A|<\ka}R_X(A),
\]
where as usual $A\leq G$ denotes that $A$ is a subgroup of $G$,
and again we take restrictions for homomorphisms.
\end{defn}

Note that if $A$ is a subgroup of $G$ then $R_X(A)\subseteq A\cap R_X(G)$,
and so $R_X^\ka(G)$ is a subgroup of $R_X(G)$.
In \cite{EdA:CCAG} (see also \cite{Ba-Ma:GR}) 
it is shown how to draw large cardinal strength from
group radical considerations. 

\begin{thm}[see \mbox{\cite[Theorem~4.11.1]{Ba-Ma:GR}}]\label{radicalstrcpct}
A cardinal $\ka$ is $\omega_1$-strongly compact if and only
if $R_\ZZ=R_\ZZ^\ka$.
\end{thm}

With this in place, our example is now remarkably straightforward.
Note that group homomorphisms take radicals to radicals, so we can consider
the category of abelian groups with their radical as a distinguished
predicate, with group homomorphisms as the morphisms.

\begin{thm}\label{radicalreversal}
Let $\Sigma=\{\cdot, R\}$ with $\cdot$ a binary operation symbol and
$R$ a unary predicate symbol.  
Let $\calA$ be the full subcategory of $\Str\Sigma$ whose objects are abelian
groups $G$ with the radical $R_\ZZ(G)$ as 
the interpretation of the predicate $R$.
If the inclusion of $\calA$ in $\Str\Sigma$ preserves $\ka$-directed
colimits for some regular cardinal $\ka$, 
then $\ka$ is $\omega_1$-strongly compact.
\end{thm}
\begin{proof}
We shall show that under the hypotheses of the theorem,
$R_\ZZ(G)=R_\ZZ^\ka(G)$ for every abelian group $G$.
So suppose that $\calA\into\Str\Sigma$ preserves $\ka$-directed colimits for
some regular cardinal $\ka$, and let $G$ be an arbitrary abelian group. 
Let $\D$ be the diagram of 
all of the $\Sigma$-structures $(A,R_\ZZ(A))$ with $A$ a subgroup of $G$
of cardinality less than $\ka$, and inclusions as the morphisms of the diagram.
The colimit of this diagram in $\Str\Sigma$ is as ever the direct limit,
which is easily seen to be $(G,R_\ZZ^\ka(G))$.
However, the colimit in $\calA$ must be $(G,R_\ZZ(G))$.
Indeed, if we temporarily forget the predicate for the radical, $G$ is clearly
the colimit, and then since radicals are respected by all group homomorphisms,
we see that $(G,R_\ZZ(G))$ satisfies the requirements to be the colimit in
$\calA$. 
Thus, colimit preservation tells us that in fact $R_\ZZ(G)=R_\ZZ^\ka(G)$,
and so by Theorem~\ref{radicalstrcpct}, $\ka$ is $\omega_1$-strongly compact.
\end{proof}

Since $h\in R_\ZZ(G)$ is $\Pi_1$ and consequently
$h\notin R_\ZZ(G)$ is $\Sigma_1$, we have that
$h\in \calI(R)\iff h\in R_\ZZ(G)$ is $\Delta_2$.
Hence, as alluded to above, 
the category $\calA$ of Theorem~\ref{radicalreversal} is $\Delta_2$ definable.

\bibliographystyle{plain}

\bibliography{ColimPres}

\end{document}